\documentclass[12pt,a4paper]{article}
\usepackage[centertags]{amsmath}
\usepackage{amsfonts}
\usepackage{amssymb}
\usepackage{amsthm}
\usepackage{color}
\usepackage{marginnote}
\usepackage{float}

\usepackage{bussproofs}

\usepackage{fullpage}
\usepackage{tikz}
\usepackage{verbatim}
\usetikzlibrary{arrows}

\newtheorem{cor}{Corollary}
\newtheorem{defi}{Definition}

\newtheorem{rem}{Remark}
\newtheorem{prop}{Proposition}
\newtheorem{lem}{Lemma}

\newtheorem{teo}{Theorem}

\usepackage{color}

\DeclareMathOperator*{\bigdotvee}{\text{\Large $\dot{\lor}$}}
\begin{document}

\title{Finite semantics for fragments of intuitionistic logic}

\author{Felipe S. Albarelli and Rodolfo C. Ertola-Biraben}

% \date{}

\maketitle

\begin{abstract}
 In 1932, G\"odel proved that there is no finite semantics for intuitionistic logic. 
 We consider all fragments of intuitionistic logic and check in each case whether a finite semantics exists. 
 We may fulfill a didactic goal, as little logic and algebra are presupposed. 
\end{abstract}

\section{Introduction}

As is well known, classical logic has a finite semantics. 
In the beginning of the twenties of the twentieth century, 
mathematicians like Kolmogorov, Glivenko, and Heyting began to study intuitionistic logic, 
at that time sometimes called ``the logic of M. Brouwer'' (see \cite {Ko}, \cite{Gl1}, and \cite{Hey}, respectively).  
The natural question arises whether also intuitionistic logic has a finite semantics. 
In 1932, G\"odel proved that there is no finite semantics for intuitionistic logic. 
In his words, he wrote that intuitionistic logic has no \emph{Realisierung mit endlich vielen Elementen} (see \cite{G1}). 
In fact, G\"odel's argument also holds for positive logic, 
that is, the conjunction-disjunction-conditional fragment of intuitionistic logic. 
That is, there is no finite semantics for the mentioned fragment. 
Shortly afterwards, in 1933, 
G\"odel himself proved that the conjunction-negation fragments of intuitionistic and classical logic coincide (see \cite{G2}).  
This implies that the conjunction-negation fragment of intuitionistic logic does have a finite semantics. 
However, G\"odel result does not hold when also considering premisses 
(just note that the Double Negation Law holds in classical logic, but not in intuitionistic logic). 
In particular, in this note we will want to know whether the conjunction-negation fragment has a finite semantics 
also when having premisses. 
In general, it is also natural to ask the same question regarding every fragment of intuitionistic logic, 
including the fragment with no connectives, which will be notated $\varnothing$.   
This we do in the present note.
All fragments appear pictorially in Figure \ref{B16}.  

The set $\mathfrak{F}$ of formulas is obtained in the usual way from the set of (propositional) letters $\Pi$, 
applying the connectives $\land, \lor, \to$, and $\neg$. 
Any subset of the connectives will be called a \emph{fragment} (of intuitionistic logic). 
Let $F$ be a fragment. 
Then, $\mathfrak{F}_F$ denotes the formulas where only the connectives in $F$ are applied. 
Note that in intuitionistic logic the given connectives are independent.  

As regards syntactic matters, 
it is possible to use the corresponding axioms of the Frege-style axiomatization of intuitionistic logic 
in the case of fragments with the conditional. 
For one version of the mentioned axiomatization see \cite[Section 11.1]{D}. 
In the case of fragments without the conditional, it is possible to use Gentzen's Natural Deduction. 
However, as we do not have $\perp$, we need to use something like the following two rules in the case of negation, 
for introduction and elimination, respectively,

\begin{prooftree}
 \AxiomC{[$\mathfrak A$]}
 \noLine
 \UnaryInfC{$\mathfrak B$}
 \AxiomC{[$\mathfrak A$]}
 \noLine
 \UnaryInfC{$\mathfrak{\neg B}$}
 \BinaryInfC{$\mathfrak{\neg A}$}
\end{prooftree}

\vspace{-.3cm}
\noindent and 

\begin{prooftree}
 \AxiomC{$\mathfrak A$}
 \AxiomC{$\mathfrak{\neg A}$}
 \RightLabel{.}
 \BinaryInfC{$\mathfrak B$}
\end{prooftree}

\noindent Gentzen's Natural Deduction rules for the other connectives may be found in \cite[p. 186]{Ge}. 
Alternatively, for fragments without the conditional, 
one may also proceed using the following rules, where $\Gamma \cup \{\varphi\} \subseteq \mathfrak{F}$:

\smallskip

(R$_\in$) If $\varphi \in \Gamma$, then $\Gamma \vdash_i \varphi$,

(R$_m$) If $\Gamma \vdash_i \varphi$ then $\Gamma, \psi \vdash_i \varphi$ (monotonicity),

(R$_t$) If $\Gamma \vdash_i \varphi$ and $\Delta, \varphi \vdash_i \psi$, then $\Gamma, \Delta \vdash_i \psi$ (cut rule),

(R$_\land$) $\Gamma \vdash_i \varphi \land \psi$ iff $\Gamma \vdash_i \varphi$ and $\Gamma \vdash_i \varphi$, 

(R$_\lor$) $\Gamma, \varphi \lor \psi \vdash_i \chi$ iff $\Gamma, \varphi \vdash_i \chi$ and $\Gamma, \psi \vdash_i \chi$, 

(R$_\to$) $\Gamma \vdash_i \varphi \to \psi$ iff $\Gamma, \varphi \vdash_i \psi$, 

(R$_{\neg I}$) If $\Gamma, \varphi \vdash_i \psi$ and $\Gamma, \varphi \vdash_i \neg \psi$, then $\Gamma \vdash_i \neg \varphi$, 

(R$_{\neg E}$) If $\Gamma \vdash_i \varphi$ and $\Gamma \vdash_i \neg \varphi$, then $\Gamma \vdash_i\psi$. 

\smallskip

% \noindent Given rule (R$_\lor$), it should be clear that, also if the conditional is not present, 
% we will only be considering the distributive cases.
\noindent When we write $\vdash_F$, for $F$ a fragment, we  mean that only the rules of the connectives in $F$ are applied.

\begin{figure} [ht]
\begin{center}

\begin{tikzpicture}

    \tikzstyle{every node}=[draw, circle, fill=white, minimum size=4pt, inner sep=0pt]

    % First, draw and connect some nodes

\draw (0,0) node (0) [label= {[label distance= 3mm]0:		\small $\varnothing		$}] {};
\draw (-4.5,1.5) node (1) [label= {[label distance= 1.2mm]182:	\small $\{\land\}		$}] {};
\draw (-1.5,1.5) node (2) [label= {[label distance= 2mm]182:	\small $\{\lor\}		$}] {};
\draw (1.5,1.5) node (3) [label= {[label distance= 3mm]0:	\small $\{\to\}			$}] {};
\draw (4.5,1.5) node (4) [label= {[label distance= 1.2mm]0:	\small $\{\neg\}		$}] {};
\draw (-5.6,3) node (5) [label= {[label distance= 1.2mm]182:	\small $\{ {\land, \lor} \}  	$}] {};
\draw (-3.4,3) node (6) [label= {[label distance= 1.2mm]182:	\small $\{ {\land, \to} \}   	$}] {};
\draw (-1.1,3) node (7) [label= {[label distance= 1.2mm]182:	\small $\{ {\lor, \to} \}    	$}] {};
\draw (1.1,3) node (8) [label= {[label distance= 3mm]0:		\small $\{ {\land, \neg} \}  	$}] {};
\draw (3.4,3) node (9) [label= {[label distance= 1.2mm]0:	\small $\{ {\lor, \neg} \}   	$}] {};
\draw (5.6,3) node (10) [label= {[label distance= 1.2mm]0:	\small $\{ {\to, \neg} \}    	$}] {};
\draw (-4.5,4.5) node (11) [label= {[label distance= 1.2mm]182:	\small $\{ {\land, \lor, \to} \}$}] {};
\draw (-1.5,4.5) node (12) [label= {[label distance= 4mm]0:	\small $\{ {\land, \lor, \neg}\}$}] {};
\draw (1.5,4.5) node (13) [label= {[label distance= 3mm]0:	\small $\{ {\land, \to, \neg} \}$}] {};
\draw (4.5,4.5) node (14) [label= {[label distance= 1.2mm]0:	\small $\{ {\lor, \to, \neg} \} $}] {};
\draw (0,6) node (15) [label= {[label distance= 3mm]0:		\small $\{ {\land, \lor, \to, \neg}$ \} }] {};
\draw (0)--(1)--(5)--(2)--(0)--(4)--(10)--(3)--(0) 
            	 (1)--(8)--(13)--(6)--(1)
            	 (2)--(7)--(14)--(9)--(2)
            	 (7)--(3)--(6)
            	 (14)--(7)--(11)--(6)--(13)--(10)--(4)--(8)--(12)
            	 (13)--(15)--(12)--(8)--(13)--(13)--(10)--(14)--(15)--(11)--(5)--(12)--(9)--(4);
\end{tikzpicture}

\end{center}
\caption{\label{B16} The sixteen fragments of intuitionistic logic}
\end{figure}
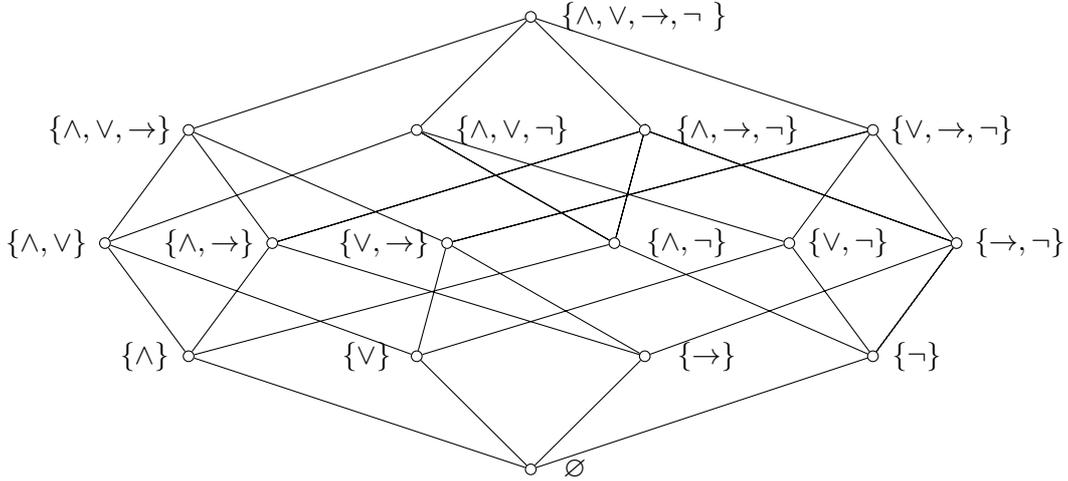

Now , let us define semantic consequence for a given algebra. 

\begin{defi}
 Let $F$ be a fragment, 
 let $\Gamma \cup \{\varphi\} \subseteq \mathfrak{F}_F$, 
 let $\mathfrak{A}_F$ be an algebra with universe $V$ and an operation for every connective in $F$, 
 and let $D \subseteq V$. 
 Then, we say that $\varphi$ is a semantic consequence of $\Gamma$ in $\mathfrak{A}_F$ with $D$ 
 (and use the notation  $\Gamma \vDash_{\mathfrak{A}_F, D} \varphi$) iff 
 for every assingment $v: \Pi \to V$, the unique homomorphism $\bar v: \mathfrak{F}_F \to V$, satisfies that 
 if $\bar v\psi \in D$, for all $\psi \in \Gamma$, then $\bar v\varphi \in D$.  
\end{defi}

We will use the same symbols for both the connectives and the corresponding algebraic operations. 
This ambiguity should not cause any problem. 

Under ``having a finite semantics'' we will understand the same as G\"odel, that is, we will use the following definition. 

\begin{defi} \label{Dfs}
 The fragment $F$ has a finite semantics iff 
 there exists an algebra $\mathfrak{A}_F$ with finite universe $V$ and $D \subseteq V$ such that 
 for every $\Gamma \cup \{ \varphi \} \in \mathfrak{F}_F$ it holds that 
 $\Gamma \vdash_F \varphi$ iff $\Gamma \vDash_{\mathfrak{A}_F, D} \varphi$.   
\end{defi}

\noindent It should be clear that, for each fragment, we will be looking for one algebra (not for a class of algebras) and, 
moreover, for a finite one.  

In Section 2, we will use G\"odel's argument in order to prove that 
any fragment having the conditional, in particular positive and intuitionistic logics, do not have a finite semantics. 
 There will only be eight fragments left to consider.  

In Section 3, we will see that the conjunction-disjunction fragments and the fragments contained in it, 
do have a finite semantics. 

In Section 4, we will consider the case of the disjunction-negation fragment. 

Funally, in Section 5, we consider the negation and conjunction-negation fragments. 

When referring to the conditional fragment, we will say $\{\to\}$-fragment. 
Analogously, in the case of other fragments.   

We think this note may fulfill a didactic goal, as little knowledge of logic and algebra are presupposed.

\section{Fragments with the conditional}

\begin{lem} \label{Lmp} 
 Let $F$ be a fragment with $\to$. Then, 
 (i) $\varphi \vdash_F \varphi$, (ii) $\vdash_F \varphi \to \varphi$, 
 and (iii) $\varphi, \varphi \to \psi \vdash_F \psi$ (\emph{modus ponens}). 
\end{lem}

\begin{proof}
 (i) follows just using (R$_\in$) and (ii) follows from (i) using (R$_\to$).  
 In order to prove (iii), use (i) to obtain $\varphi \to \psi \vdash_F \varphi \to \psi$ 
 and then apply (R$_\to$) with $\Gamma = \{ \varphi \to \psi \}$. 
\end{proof}

In what follows we will use G\"odel's argument with formulas of a different form. 
In order to do that, we will use the following abbreviation: 
$$\varphi \dot \lor \psi = (\varphi \to \psi) \to \psi.$$ 

\noindent For example, the formula $[(p_3 \to p_2) \dot \vee (p_3 \to p_1)] \dot \vee (p_2 \to p_1)$ denotes the formula 

$$([([(p_3 \to p_2) \to (p_3 \to p_1)] \to (p_3 \to p_1)) \to (p_2 \to p_1)] \to (p_2 \to p_1)).$$  

Note that $\dot \vee$ is neither commutative nor associative. 
We will omit parentheses supposing associativity to the left. 
So, instead of the given formula, we may as well write 

$$(p_3 \to p_2) \to (p_3 \to p_1) \to (p_3 \to p_1) \to (p_2 \to p_1) \to (p_2 \to p_1).$$

We will use the following lemma. 

\begin{lem} \label{Ldv}
 Let $F$ be a fragment with $\to$ and $\varphi \in \mathfrak{F}_F$ such that $\vdash_F \varphi$. 
 Then (i) $\vdash_F \varphi \dot \vee \psi$, 
 (ii) $\vdash_F \psi \dot \vee \varphi$, for any formula $\psi \in \mathfrak{F}_F$, and
 (iii) If $\psi = \cdots \dot \lor \ \varphi \ \dot \lor \cdots$, 
 where the given dots may be empty at the beginning or the end, 
 then $\vdash_F \psi$. 
\end{lem}

\begin{proof}
 (i) By (mp) we have $\varphi, \varphi \to \psi \vdash_F \psi$. 
 Then, using (R$_\to$), it follows that $\varphi \vdash_F (\varphi \to \psi) \to \psi$. 
 As we have $\vdash_F \varphi$, using the cut rule it follows that $\vdash_F (\varphi \to \psi) \to \psi$. 
 (ii) As we have $\vdash_F \varphi$, by (R$_m$) it follows that $\psi \to \varphi \vdash_F \varphi$. 
 Then, by (R$_\to$) it follows that $\vdash_F (\psi \to \varphi) \to \varphi$. 
 Part (iii) follows using (i) and (ii). 
\end{proof}

We will also use the following algebraic facts, using the $\dot \vee$ notation in a way analogous to the logical case. 

\begin{lem} \label{Lid}
 Let $\mathbf{A}$ be a Heyting algebra with universe $A$. 
 Let $a, b \in A$. 
 Then, (i) If $\mathbf{A}$ is a chain and $a < b$, then $b \to a = a$ and  
 (ii) if $a \leq b$, then $a \dot \lor b = b$.
\end{lem}

\begin{proof}
 (i) It is clear that (1) $b \wedge a \leq a$. 
 Now, let us suppose that $b \wedge c \leq a$, for any $c \in A$. 
 Then, as $A$ is a chain, then either $b \wedge c = b$ or $b \wedge c = c$. 
 Now, as $a < b$, it cannot be the case that $b \wedge c = b$. 
 So, $b \wedge c = c$. 
 Then $c \leq a$. 
 So, we have that, (2) for any $c \in A$, if $b \wedge c \leq a$, then $c \leq a$. 
 From (1) and (2) it follows that $b \to a = a$.  
 
 (ii) Let us suppose that $a \leq b$. So, $a \to b$ is top. 
 So, $(a \to b) \to b \leq b$. 
 It is also the case that $b \leq (a \to b) \to b$. 
\end{proof}

Next come two propositions. 

\begin{prop} \label{Pdfc}
 Let $F$ be a fragment with $\to$ such that $F$ has a finite semantics, say with $n \geq 1$ values. 
 Then, the formulas of the following form are derivable in $F$: 
 $$\alpha_n = \bigdotvee_{1 \leq i < j \leq n+1} p_j \to p_i.$$
\end{prop}

\begin{proof}
 Let us suppose that $F$ is a fragment with $\to$ that has a semantics with $n$ values, that is, 
 there exists an algebra with universe $|V|=n$ and $D \subseteq V$ such that 
 for every $\Gamma \cup \{ \varphi \} \in \mathfrak{F}_F$ it holds that 
 $$\textrm{(C)\ \ }\Gamma \vdash_F \varphi\ \textrm{iff} \ \Gamma \vDash_{\mathfrak{A}_F, D} \varphi.$$  
 Let us take a valuation $w:\mathfrak{F}_F \to V$ and let us consider $w(\alpha_n)$. 
 As there are $n+1$ propositional letters in $\alpha_n$, but only $n$ values, 
 there must be letters $p_i$, $p_j$ such that $w(p_i) = w(p_j)$. 
 Now, let us consider the formula $\beta_n = \alpha_n[p_i/p_j]$.  
 It should be clear that $w(\beta_n) = w(\alpha_n)$. 
 Now, $\beta_n = \cdots \dot \lor (p_j \to p_j) \dot \lor \cdots$, 
 where the given dots may be empty at the beginning or the end.   
 Now, using Lemma \ref {Lmp} (ii), it holds that $\vdash_F p_j \to p_j$ and then, 
 using Lemma \ref{Ldv} (iii), $\vdash_F \beta_n$. 
 Then, by (C), $\vDash_{\mathfrak{A}_F, D} \beta_n$. 
 So, $\vDash_{\mathfrak{A}_F, D} \alpha_n$. 
 So, using (C) in the other direction, $\vdash_F \alpha_n$.
\end{proof}

\begin{prop} \label{Pndfc}
 Let $\alpha_n$ be a formula as in Proposition \ref{Pdfc}. 
 Then, $\nvdash_i \alpha_n$, for any natural number $n \geq 1$. 
\end{prop}

\begin{proof}
 Let us consider the $n+1$-element chain of the first $n+1$ natural numbers with the usual order. 
 Defining meet, join, relative meet complement and meet complement as usual, 
 the given chain is a Heyting algebra. 
 Let us consider any assignment $w$ such that $w(p_i)=i$, for  $1 \leq i \leq n+1$. 
 Then,  
 \begin{align*}
  \bar w(\alpha_n) & = \bigdotvee_{1 \leq i < j \leq n+1} \bar wp_j \to \bar wp_i, \\
  & = \bigdotvee_{1 \leq i < j \leq n+1} \bar wp_i, \textrm{(as} \ \bar wp_i < wp_j\textrm{, using Lemma \ref{Lid}(i)),} \\
  & = \bar wp_1 \dot \lor \bar wp_1 \dot \lor \cdots \dot \lor wp_2 \dot \lor wp_2 \cdots \dot \lor wp_n, \\
  & = \bar wp_1 \dot \lor \bar wp_2 \dot \lor \cdots \dot \lor wp_n \textrm{, (by Lemma \ref{Lid}(ii)),} \\
  & = \bar wp_n \textrm{, (by Lemma \ref{Lid}(ii)),}\\
  & = n, \\
  & \neq n+1.
 \end{align*}
 Using soundness, it follows that $\nvdash_i \alpha_n$. 
\end{proof}

\begin{teo}
 Fragments containing $\to$ do not have a finite semantics.  
\end{teo}

\begin{proof}
 Applying Proposition \ref{Pdfc}, the formulas $\alpha_n$ would be derivable in the fragment, 
 which cannot be the case, as by Proposition \ref{Pndfc}, $\nvdash_i \alpha_n$.    
\end{proof}

\section{Subfragments of $\{ \land, \lor \}$}

Due to the results of the previous section, 
the remaining fragments to be considered are the ones appearing in Figure \ref{B8}. 

\begin{figure} [ht]
\begin{center}

\begin{tikzpicture}

    \tikzstyle{every node}=[draw, circle, fill=white, minimum size=4pt, inner sep=0pt, label distance=0.5mm]

    % First, draw and connect some nodes

    \draw (0,0) node (0) [label= right : \small 	$\varnothing$]{};
    \draw (-1.4,1.4) node (1) [label = left :\small 	$\{\land \}$]{};
    \draw (0,1.4) node (2) [label = right :\small 	$\{ \lor \}$]{};
    \draw (1.4,1.4) node (3) [label = right :\small 	$ \{\neg \} $]{};
    \draw (-1.4,2.8) node (4) [label = left :\small 	$\{ {\land, \lor} \}$]{};
    \draw (0,2.8) node (5) [label = right :\small 	$\{ {\land, \neg} \}$]{};
    \draw (1.4,2.8) node (6) [label = right :\small 	$\{ {\lor, \neg} \}$]{};
    \draw (0,4.2) node (7) [label = right :\small 	$\{ {\land, \lor, \neg} \}$]{};
    \draw (0)--(1)--(4)--(2)--(0)--(3)--(6)--(2)
    	  (1)--(5)--(3) (4)--(7)--(6)  (5)--(7);
 
\end{tikzpicture}

\end{center}
\caption{\label{B8} The fragments without the conditional}
\end{figure}
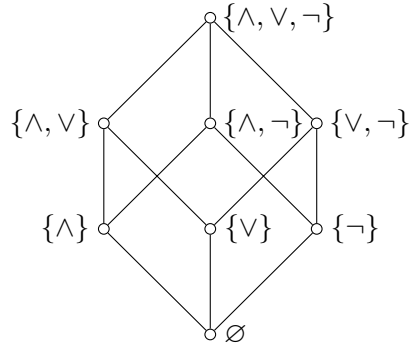

In this section, $\mathbf 2$ stands for any algebra  of the form $\langle \{0, 1\}; F \rangle$, 
where $F \subseteq \{\wedge, \vee \}$ and $\land$ and $\lor$ stand for the usual meet and join in a lattice or semilattice. 
In all cases, $D=\{ 1 \}$. 
It will be clear in the context which is the relevant $F$, which may be the empty set. 

\begin{prop} \label{Pcf}
 Let $\Gamma \cup \{ \varphi \} \subseteq \mathfrak{F}_{\{\land \}}$. 
 Then, $\Gamma \vdash_{\{\land \}} \varphi$ iff $\Gamma \vDash_\mathbf 2 \varphi$. 
\end{prop} 

\begin{proof}
 If $\Gamma \vdash_{\{\land \}} \varphi$, then $\Gamma \vdash_i \varphi$. 
 Using soundness of intuitionistic logic, it follows that $\Gamma \vDash_\mathbf 2 \varphi$. 
 On the other hand, suppose $\Gamma \nvdash_{\{\land \}} \varphi$, 
 that is, $\Gamma \nvdash_{\{\land \}} p_1 \land p_2 \land \cdots \land p_n$. 
 Then, by (R$_\wedge$), there is a letter $p_i$ such that $\Gamma \nvdash_{\{\land \}} p_i$. 
 So, again by (R$_\wedge$), $p_i$ is not a subformula of any formula in $\Gamma$. 
 Then, there exists the valuation $w$ such that $wp_i = 0$ and $wp=1$ for letters $p$ other that $p_i$. 
 So, $\Gamma \nvDash_ \mathbf 2 \varphi$.
\end{proof}

\begin{cor}
 Let $\Gamma \cup \{ \varphi \} \subseteq \mathfrak{F}_\varnothing$. 
 Then, $\Gamma \vdash_\varnothing \varphi$ iff $\Gamma \models_ \mathbf 2 \varphi$.
\end{cor}

\begin{prop} \label{Pcdf}
 Let $\Gamma \cup \{ \varphi \} \subseteq \mathfrak{F}_{\{\land, \lor\}}$. 
 Then, $\Gamma \vdash_{\{\land, \lor \}} \varphi$ iff $\Gamma \vDash_\mathbf 2 \varphi$. 
\end{prop}

\begin{proof}
 If $\Gamma \vdash_{\{\land, \lor \}} \varphi$, then $\Gamma \vdash_i \varphi$. 
 Using soundness of intuitionistic logic, it follows that $\Gamma \vDash_\mathbf 2 \varphi$. 
 On the other hand, suppose $\Gamma \nvdash_{\{\land, \lor \}} \varphi$. 
 Then using the conjunctive normal form theorem, 
 it follows that $\varphi$ and every formula in $\Gamma$ may be seen as a conjunction of disjunction of letters. 
 Then, by (R$_\land$), $\Gamma \nvdash_{\{\land, \lor \}} \chi$, where $\chi = q_1 \lor \cdots \lor q_n$. 
 Also, as every formula in $\Gamma$ is a conjunction (of disjunctions), 
 and to have formulas $\alpha$ and $\beta$ as different premisses is equivalent to having $\alpha \land \beta$ as only premiss, 
 then we might as well consider $\Gamma$ to be a set of disjunctions and call it $\Delta$.  
 Now, by (R$_\lor$), it follows that 
 every disjunction in $\Delta$ has a letter that does not belong to the set $\{q_1, \dots q_n \}$. 
 Consequently, 
 there exists the valuation $w$ such that $wq_i = 0$, for all $1 \leq i \leq n$ and $wp = 1$ for letters $p$ other than the $q_i$. 
 So, every formula in $\Delta$ will have value $1$. 
 So, $\Gamma \nvDash_2 \varphi$. 
\end{proof}

\begin{cor}
 Let $\Gamma \cup \{ \varphi \} \subseteq \mathfrak{F}_{\{\lor\}}$. 
 Then, $\Gamma \vdash_{\{\lor \}} \varphi$ iff $\Gamma \vDash_\mathbf 2 \varphi$. 
\end{cor}

\begin{proof}
 The proof for the fragment $\{ \land, \lor \}$ was reduced to having only disjunctions. 
\end{proof}

\section{The disjunction-negation fragment}

Due to the results of the previous sections, 
the remaining fragments to be considered are the ones appearing in Figure \ref{B4}. 

\begin{figure} [ht]
\begin{center}

\begin{tikzpicture}

    \tikzstyle{every node}=[draw, circle, fill=white, minimum size=4pt, inner sep=0pt, label distance=1mm]

    % First, draw and connect some nodes

    \draw (0,0) node (0) [label= right :  \small$\{ \neg\}$]{};
    \draw (-1,1) node (1) [label = left : \small$\{ {\land, \neg} \}$]{};
    \draw (0,2) node (2) [label = right : \small$\{ {\land, \vee, \neg} \}$]{};
    \draw (1,1) node (3) [label = right : \small$ \{ {\vee, \neg} \} $]{};
    \draw (0)--(1)--(2)--(3)--(0);
 
\end{tikzpicture}

\end{center}
\caption{\label{B4} The four remaining fragments}
\end{figure}
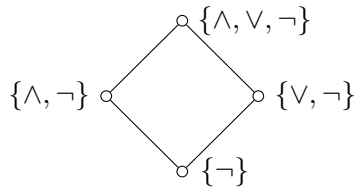

We will use the following lemmas. 

\begin{lem} \label{Ldi}
 Let $F$ be a fragment with $\lor$. 
 Then, $\varphi \vdash_F \varphi \lor \psi$ and $\psi \vdash_F \varphi \lor \psi$. 
\end{lem}

\begin{proof}
 Using (R$_\in$), we get $\varphi \vee \psi \vdash_F \varphi \vee \psi$. 
 Then, using (R$_\vee$) with $\Gamma = \varnothing$, 
 we get both $\varphi \vdash_F \varphi \lor \psi$ and $\psi \vdash_F \varphi \lor \psi$. 
\end{proof}

\begin{lem} \label{Ldn}
 Let $F$ be a fragment with $\lor$ and $\neg$. 
 Then, (i) if $\varphi \vdash_F \psi$, then $\neg \psi \vdash_F \neg \varphi$ and 
 (ii) $\vdash_F \neg \neg (\varphi \vee \neg \varphi)$. 
\end{lem}

\begin{proof}
  (i) Suppose $\varphi \vdash_F \psi$. 
  Then, by monotonicity, we have $\varphi, \neg \psi \vdash_F \psi$. 
  By (R$_\in$) we also have $\varphi, \neg \psi \vdash_F \neg \psi$. 
  Then, using (R$_{\neg I}$), it follows that $\neg \psi \vdash_F \neg \varphi$. 
  (ii) Using Lemma \ref{Ldi}, we have $\varphi \vdash_F \varphi \vee \neg \varphi$. 
  Then, by part (i), $\neg (\varphi \vee \neg \varphi) \vdash_F \neg \varphi$. 
  Using Lemma \ref{Ldi} again, we have $\neg \varphi \vdash_F \varphi \vee \neg \varphi$. 
  So, using the cut rule, $\neg (\varphi \vee \neg \varphi) \vdash_F \varphi \vee \neg \varphi$. 
  Now, by (R$_\in$), we also have $\neg (\varphi \vee \neg \varphi) \vdash_F \neg (\varphi \vee \neg \varphi)$. 
  Finally, using (R$_{\neg I}$), we get $\vdash_F \neg \neg (\varphi \vee \neg \varphi)$. 
\end{proof}

\begin{rem}
 Recall that \emph{Tertium non datur} does not hold in intuitionistic logic. 
However, as seen in part (ii) of the last Lemma, its double negation holds in any fragment with $\vee$ and $\neg$. 
\end{rem}

\noindent On the other hand, intuitionistic logic enjoys the Disjunction Property, 
which does not hold for classical logic.  

\begin{lem} \label{Ldp}
 Let $\varphi, \psi \in \mathfrak{F}$.  
 If $\vdash_i \varphi \vee \psi$, then $\vdash_i \varphi$ or $\vdash_i \psi$. 
\end{lem}

\begin{proof}
 An algebraic proof runs as follows. 
 If neither $\vdash_i \varphi$ nor $\vdash_i \psi$ hold, 
 then there are Heyting algebras $\mathbf{H_1}$, $\mathbf{H_2}$ and valuations $v_1$, $v_2$ such that 
 $v_1 \varphi \neq 1_{\mathbf{H_1}}$ and $v_2 \psi \neq 1_{\mathbf{H_2}}$. 
 Now, take the direct product $\mathbf{H_1} \times \mathbf{H_2}$ and 
 add an element which is greater than any element of the universe of the given product. 
 Then, the resulting algebra with the natural valuation will prove that 
 it is not the case that $\vdash_i \varphi \vee \psi$. 
 For details, the reader may see \cite{S}.  
 For other proofs, see \cite[Exercise 2.6.7 or sections 5.6 to 5.10] {TvD}. 
\end{proof}

\begin{prop} \label{Pdfdn}
 Let $F$ be such that \{$\lor, \neg \} \subseteq F$. 
 Let us suppose that $F$ has a finite semantics, say with $n \geq 1$ values. 
 Then, the formulas of the following form are derivable in $F$: 
 
 $$\alpha_n = \bigvee_{1 \leq i < j \leq n+1} \neg \neg (\neg p_i \lor p_j).$$
 
\end{prop}

\begin{proof}
 Let us suppose that $F$ is a fragment with $\lor$ and $\neg$ that has a semantics with $n$ values, that is, 
 there exists an algebra with universe $V$ such that $|V|=n$ and $D \subseteq V$ such that 
 for every $\Gamma \cup \{ \varphi \} \in \mathfrak{F}_F$ it holds that 
 $$\textrm{(C)\ \ }\Gamma \vdash_F \varphi\ \textrm{iff} \ \Gamma \vDash_{\mathfrak{A}_F, D} \varphi.$$  
 Let us take a valuation $w:\mathfrak{F}_F \to V$ and let us consider $w(\alpha_n)$. 
 As there are $n+1$ propositional letters in $\alpha_n$, but only $n$ values, 
 there must be letters $p_i$, $p_j$ such that $w(p_i) = w(p_j)$. 
 Now, let us consider the formula $\beta_n = \alpha_n[p_i/p_j]$.  
 It should be clear that $w(\beta_n) = w(\alpha_n)$. 
 Now, it holds that $\vdash_F \neg \neg (\neg p_j \lor p_j)$. 
 Consequently, $\vdash_F \beta_n$. 
 Then, by (C), $\vDash_{\mathfrak{A}_F, D} \beta_n$. 
 So, $\vDash_{\mathfrak{A}_F, D} \alpha_n$. 
 So, using (C) in the other direction, $\vdash_F \alpha_n$.
\end{proof}

\begin{prop} \label{Pndfdn}
 The formulas of the form given in Proposition \ref{Pdfdn} are not intuitionistically derivable. 
\end{prop}

\begin{proof}
 For every $i$, $j$, $i \neq j$, $\neg \neg (\neg p_i \lor p_j)$ is not even classically derivable. 
 So, by Lemma \ref{Ldp}, it follows that $\nvdash_i \alpha_n$. 
\end{proof}

\begin{teo}
 Fragments containing $\lor$ and $\neg$ do not have a finite semantics.  
\end{teo}

\begin{proof}
 Applying Proposition \ref{Pdfdn}, the formulas of the given form would be derivable in the fragment, 
 which cannot be the case, as they are not intuitionistically derivable, as stated in Proposition \ref{Pndfdn}.    
\end{proof}

\section{The negation and conjunction-negation fragments}

We only need to consider fragments $\{\neg \}$ and $\{\land, \neg \}$. 
In the Introduction we stated that 
G\"odel proved that the set of derivable formulas of the conjunction-negation fragment of intuitionistic logic 
coincides with the set of classically derivable formulas. 
This is also stated and proved in detail in \cite{Kl} (see Corollary to (a2) in p. 493). 
This implies that the conjunction-negation fragment has a finite semantics with respect to derivable formulas, 
that is, two-valued classical semantics.   
The natural question arises whether we also have a finite semantics when having premisses as well. 
This we solve in this section. 

As regards syntactics, in this section we will use the following version of the celebrated Glivenko Theorem 
and also the given Corollary. Glivenko Theorem was originally proved for intuitionistic logic in \cite{Gl2}. 
Before stating those facts, 
we say that a set of formulas $\Gamma \in \mathfrak{F_{\land, \neg}}$ is classically 
(respectively $\{\land, \neg \}$-) consistent iff
from $\Gamma$ we may not arrive to a contradiction in classical logic 
(respectively in the $\{\land, \neg \}$-fragment of intuitionistic logic), 
where by a contradiction we mean a pair $\varphi, \neg \varphi$ of formulas.  

\begin{teo} \label{GT}
 Let $\Gamma \cup \{ \varphi \} \subseteq \mathfrak{F_{\land, \neg}}$. 
 Then, if $\Gamma \vdash_c \neg \varphi$, then $\Gamma \vdash_{\{ \land, \neg \}} \neg \varphi$.
\end{teo}

\begin{cor} \label{GTC}
 Let $\Gamma \subseteq \mathfrak{F_{\land, \neg}}$.
 If $\Gamma$ is $\{\land, \neg\}$-consistent, then $\Gamma$ is classically consistent.  
\end{cor}

Regarding semantics, we will use the concepts of subalgabra and congruence, which we now state 
(for details or examples the reader may see \cite{BS}).

\begin{defi}
 Given two algebras $\mathbf{A}$ and $\mathbf{B}$ of the same type, 
 we say that $\mathbf{B}$ is a subalgebra of $\mathbf{A}$ iff 
 the universe of $\mathbf{B}$ is included in the universe of $\mathbf{A}$ and 
 every fundamental operation of $\mathbf{B}$ 
 is the restriction to the universe of $\mathbf{B}$ of the corresponding operation of $\mathbf{A}$. 
\end{defi}

\begin{defi}
 Given an algebra $\mathbf{A} = \langle A; F\rangle$, 
 a congruence on $\mathbf{A}$ is an equivalence relation $E$ on $A$ such that 
 for every n-ary operation $f$ in $F$ and elements $a_i, b_i$ in $A$, 
 
 if $a_i E b_i$, for all $i, 1 \leq i \leq n$, then $f(a_1, \dots, a_n) E f(b_1, \dots, b_n)$.  
\end{defi}

\noindent The diagonal relation and the all relation are the only trivial congruences. 

In this section, $\mathbf 3$ will stand for the algebra $\langle\{0, \,^1\!/_2, 1 \}; F\rangle$, 
where $F$ is either $\{\wedge, \neg \}$ or $\{\neg \}$. 
and $\land$ and $\neg$ stand for the usual meet and meet complement in a Heyting algebra 
(we might as well say that $a \wedge b=$ min $\{a, b\}$, for any $a, b$ in the universe of $\mathbf{3}$, 
$\neg 0 = 1$, and $\neg ^1\!/_2 = \neg 1 = 0$). 
In all cases, $D=\{ 1 \}$. 
It will be clear in the context which is the relevant $F$. 
The algebra $\mathbf{3}$ appears in Figure \ref{3} together with the only non-trivial congruence given by the ellipses.  
Note, also, that  $\langle\{0, 1 \}; \wedge, \neg \rangle$, 
where $\wedge$ and $\neg$ are the usual meet and meet complement, 
is a subalgebra of $\mathbf{3}$.

\begin{figure} [ht]
\begin{center}

\begin{tikzpicture}

    \tikzstyle{every node}=[draw, circle, fill=white, minimum size=4pt, inner sep=0pt, label distance=1mm]

    \draw (0,0) node (0)	[label= {[label distance= 4mm]0:\small $0$}]{};
    
    \draw (0,1) node (1/2)	[label= {[label distance= 4mm]0:\small $^1\!/_2$}]{};
    \draw (0,2) node (1)	[label= {[label distance= 4mm]0:\small $1$}]{};
    \draw (0)--(1/2)--(1);
 
    \draw (0,1.5) ellipse (.4 and .9);  
    \draw (0,0) ellipse (.4 and .4);
\end{tikzpicture}

\end{center}
\caption{\label{3} The algebra $\mathbf{3}$ with its only non-trivial congruence}
\end{figure}
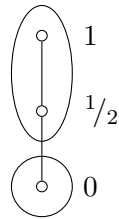

Let us now consider the fragment $\{\land, \neg \}$.  

\begin{prop} \label{Pcnf}
 Let $\Gamma \cup \{ \varphi \} \subseteq \mathfrak{F}_{\{\land, \neg\}}$. 
 Then, $\Gamma \vdash_{\{\land, \neg\}} \varphi$ iff $\Gamma \vdash_{\bf{3}} \varphi$. 
\end{prop}

\begin{proof}
 If $\Gamma \vdash_{\{\land, \neg\}} \varphi$, then $\Gamma \vdash_i \varphi$. 
 Using soundness of intuitionistic logic, it follows that $\Gamma \models_{\bf{3}} \varphi$. 
 On the other hand, suppose $\Gamma \nvdash_{\{\land, \neg\}} \varphi$. 
 There are three cases: (i) $\varphi = \varphi_1 \land \varphi_2 \land \cdots \land \varphi_n$, 
 where each $\varphi_i$ is a either a negation or a letter, 
 (ii) $\varphi = \neg \psi$, for some formula $\psi$, or (iii) $\varphi = p$, for some letter $p$. 
 In case (i), by (R$_\land$), $\Gamma \nvdash_{\{\land, \neg\}} \varphi_1 \land \varphi_2 \land \cdots \land \varphi_n$ iff 
 there is at least an $i$ such that $\Gamma \nvdash_{\{\land, \neg\}} \varphi_i$ such that  
 $\varphi_i$ is either a negation or a letter. 
 So, this case reduces to either case (ii) or case (iii). 
 In case (ii), using Theorem \ref{GT}, we have that $\Gamma \nvdash_c \neg \psi$. 
 Then, there is a valuation $v$ in the two element Boolean algebra 
 such that $v\psi=1$ for all $\psi \in \Gamma$ and $v\neg \varphi =0$.    
 Then, as the Boolean algebra of two elements with fundamental operations $\land$ and $\neg$ is a subalgebra of $\mathbf{3}$, 
 it follows, using the same valuation $v$, that $\Gamma \nvDash_{\bf{3}} \varphi$. 
 Finally, case (iii) means that we have $\Gamma \nvdash_{\{\land, \neg \}} p$. 
 Then, $\Gamma$ is $\{\land, \neg\}$-consistent. 
 Then, using Corollary \ref{GTC}, it follows that $\Gamma$ is classically consistent. 
 Then, there exists a assignment $v$ such that in the two-element Boolean algebra 
 we have that $\bar v\psi=1$ for all $\psi \in \Gamma$. 
 Now, let us define an assignment $w:\Pi \to \mathbf{3}$ such that $wp_i = vp_i$, for all $p_i \in \Pi$. 
 Then, as the given Boolean algebra is a subalgebra of $\mathbf{3}$, $\bar w\psi = 1$ for all $\psi \in \Gamma$. 
 If $\bar wp=0$, then we are done.  
 In case $\bar wp=1$, let us define $w'$ like $w$ except for $w'p = \,^1\!/_2$. 
 Now, as $\Gamma \nvdash_{\{\land, \neg\}} p$, 
 then, due to (R$_\in$) and (R$_\land$), $p$ can only appear as subfomula of a formula $\psi$ in $\Gamma$ 
 if it appears in the scope of a negation, and the value of $\psi$ will not change, 
 as $\langle(1,1), (1,^1\!/_2), (^1\!/_2,^1\!/_2), (^1\!/_2,1), (0,0)\rangle$ is a congruence relation. 
\end{proof}

Now we can easily deal with the negation fragment. 

\begin{prop} \label{Pnf}
 Let $\Gamma \cup \{ \varphi \} \subseteq \mathfrak{F}_{\{\neg\}}$. 
 Then, $\Gamma \vdash_{\{\neg\}} \varphi$ iff $\Gamma \models_{\bf{3}} \varphi$. 
\end{prop}

\begin{proof}
 The left to right direction is the same as in the previous proof. 
 For the other direction, just consider cases (ii) and (iii) in the previous proof. 
\end{proof}

It is clear, then, due to propositions \ref{Pcnf} and \ref{Pnf}, 
that the fragments $\{\land, \neg \}$ and $\{ \neg \}$ have finite semantics. 

A final remark. We have seen that we do not have finite semantics for many fragments of intuitionistic logic.  
However, the Finite Model Property (FMP) holds for intuitionistic logic (so, also, for any of its fragments). 
The difference may be understood as the difference between the quantifications $\exists \forall$ and $\forall \exists$, that is, 
we do not have a finite semantics for all cases. 
Now, given any particular case that does not hold, 
we may find a finite interpretation (for example, a Heyting algebra) that proves that it is not the case.  
The reader interested in a proof of the FMP for intuitionistic logic may see \cite[Section 11.9 ]{D}.

\end{document}